\documentclass[11pt]{article}

\usepackage[margin=1in]{geometry}
\usepackage[T1]{fontenc}
\usepackage[utf8]{inputenc}
\usepackage{lmodern}
\usepackage{amsmath,amssymb,amsthm,mathtools}
\usepackage{empheq}
\usepackage{bm}
\usepackage{microtype}
\usepackage{hyperref}

\numberwithin{equation}{section}

\newtheorem{theorem}{Theorem}[section]
\newtheorem{lemma}[theorem]{Lemma}
\newtheorem{proposition}[theorem]{Proposition}
\newtheorem{corollary}[theorem]{Corollary}

\theoremstyle{definition}
\newtheorem{definition}[theorem]{Definition}

\theoremstyle{remark}

\title{Spinorial Div-Curl Structure and Bilinear Null-Form Estimates for Dirac Equations}

\author{ 
   Jiong-Yue Li \thanks{School of Mathematics, Sun Yat-sen University, Guangzhou, China. 
  Email: \texttt{lijiongyue@mail.sysu.edu.cn}}
}

\date{}

\begin{document}

\maketitle

\begin{abstract}
We identify a physical-space balance-law mechanism underlying bilinear 
null forms for the free Dirac equation in three space dimensions. For each 
spatial direction, we decompose a spinor into the two eigenspaces of the 
directional Dirac symbol. The principal parts of the corresponding modes 
propagate in opposite directions, while the transverse derivatives and the 
mass term couple them. After integration over the transverse variables, 
their charge densities satisfy a pair of one-dimensional balance laws, and 
a div--curl interaction estimate controls the mixed product of these 
densities. The algebraic anticommutation condition defining the spinorial 
null form exchanges exactly the same two eigenspaces. This identifies the 
algebraic cancellation with the interaction selected by the balance laws. 
Combined with angular localization, the argument yields frequency-localized 
$L^2$ spacetime estimates while preserving the natural first-order 
formulation of the Dirac equation. In the massless case, the estimate has 
the same lower-frequency scaling as the three-dimensional wave null-form 
estimate and gains half a derivative over the direct product bound. In the 
massive case, we obtain a channel-dependent refinement. For the 
pseudoscalar interaction, anticommutation with the full massive Dirac 
Hamiltonian gives an additional frequency-to-mass factor $K/m$ for 
low-frequency interactions within the same energy branch, where $K$ is 
the frequency scale and $m$ is the mass. This gain is absent in the scalar 
channel. The bilinear estimates also yield factorized bounds for cubic 
spinorial null forms arising in nonlinear Dirac models.
\end{abstract}


% Add the appropriate 2020 Mathematics Subject Classification codes here.

\section{Introduction}
The div--curl principle is a basic mechanism for extracting hidden
cancellation from differential constraints. It originated in the theory of
compensated compactness developed by Murat and Tartar
\cite{Murat1978,Tartar1979}. In its classical form, the div--curl lemma shows that,
although two bounded sequences may converge only weakly, the product of a
divergence-controlled sequence and a curl-controlled sequence can still pass
to the limit. Coifman, Lions, Meyer and Semmes later obtained an endpoint
version in harmonic analysis, showing that the product of an $L^2$
divergence-free vector field and an $L^2$ curl-free vector field belongs to
the Hardy space $\mathcal H^1$ \cite{CLMS1993}. These results demonstrate that
differential constraints may compensate for the lack of strong convergence
or integrability.

A closely related cancellation appears in nonlinear wave equations through
null forms. The standard null forms vanish when two waves propagate in the
same null direction, and hence their products satisfy better spacetime
estimates than general quadratic terms. Klainerman and Machedon proved the
fundamental $L^2_{t,x}$ estimates for wave null forms using Fourier analysis
and the geometry of the light cone \cite{KlainermanMachedon}. Klainerman,
Rodnianski and Tao subsequently developed a physical-space approach based on
wave packets, characteristic directions and induction on scales
\cite{KRT}. More recently, Wang and Zhou introduced a quantitative
div--curl type estimate for a pair of one-dimensional balance laws
\cite{Zhou}\cite{WangZhou}. Their argument is based on a simple interaction functional
and converts conservation or balance laws directly into bilinear spacetime
bounds. In particular, this method provides a short physical-space proof of
frequency-localized wave null-form estimates and has also been applied to
several nonlinear wave and dispersive equations \cite{WangZhou-Schordinger,LaiShaoZhou,ZhangZhou-1,
ShaoZhou,WangZhou-Hypersurface,TuZhou,HuTuZhou,ZhangZhou-2}.

Null structures also play a central role in the analysis of nonlinear Dirac
equations. They have been identified and used in the Dirac-Klein-Gordon
system and in several cubic Dirac models
\cite{DAFS,BejenaruHerr,BournaveasCandy,KatayamaKubo}\cite{WangXuecheng}\cite{LiZang-1,LiZang-2}. Most previous
approaches describe the cancellation through the spectral projections of
the Dirac operator and then use angular decompositions, bilinear Fourier
restriction estimates, null-frame spaces or atomic function spaces. These
methods are very effective, but the connection between the spinorial
algebra and a div--curl interaction in physical space is less explicit.

The purpose of this paper is to give such a connection. We develop a spinorial div-curl mechnism 
for the free Dirac equation in $1+3$ dimensions. By decomposing the Dirac charge current into the two eigenspaces of the directional symbol $\alpha\cdot\omega$ and integrating out the transverse variables, we derive a pair of one-dimensional balance laws for oppositely propagating spinorial modes. The associated div-curl interaction controls a class of Dirac null forms characterized by the anticommutation relation $\{\Gamma, \alpha\cdot\omega\}=0$, yielding frequency-localized $L_{t,x}^{2}$ bilinear estimates with the same low-frequency gain as the three-dimensional wave null-form estimate. We further distinguish a genuinely massive null structure: for $\Gamma=\beta\gamma^{5}$, which anticommutes with the full massive Dirac-Hamiltonian, low-frequency same-branch interactions gain an additional factor $K/m$. The bilinear estimates also yield factorized bounds for cubic spinorial null forms arising in nonlinear Dirac models \cite{Soler1970}\cite{Thirring1958}.

\medskip

Let $\Phi:\mathbb{R}^{1+3}\rightarrow\mathbb{C}^{4}$ solve the Dirac equation
\begin{equation}\label{Dirac-1}
\begin{cases}
   -i \gamma^{\mu}\partial_{\mu}\Phi+m\Phi=0\\
    \Phi(0,x)=\Phi_{0},
\end{cases}    
\end{equation}
where $m$ denotes the mass of the particle. The standard gamma matrices are given by
\[  \gamma^{0}=\begin{pmatrix}
                               I_{2}   & 0 \\
                               0       &  -I_{2}
                          \end{pmatrix}, \qquad  \gamma^{j}=\begin{pmatrix}
                                                                                      0  &    \sigma^{j}\\
                                                                                      -\sigma^{j}  &  0
                                                                                      \end{pmatrix},
\]
where $\sigma^{j}$ $(j=1,2,3)$ denotes the Pauli matrices
\[
     \sigma^{1}=\begin{pmatrix}
                         0 & 1\\
                         1 & 0
                         \end{pmatrix},\qquad 
      \sigma^{2}=\begin{pmatrix}
                         0 & -i\\
                         i & 0
                         \end{pmatrix},\qquad 
        \sigma^{3}=\begin{pmatrix}
                         1 & 0\\
                         0 & -1
                         \end{pmatrix}.
\]
Gamma matrices satisfy the Clifford relation $\gamma^{\mu}\gamma^{\nu}+\gamma^{\nu}\gamma^{\mu}=-2\eta^{\mu\nu}I$, $(\mu,\nu=0,1,2,3)$. Notice that $\gamma^{0}\gamma^{0}=I$ and $\gamma^{j}\gamma^{j}=-I$. With $\beta=\gamma^{0}$, $\alpha^{j}=\gamma^{0}\gamma^{j}$, and $D=-i\nabla$, the equation can be written in Hamiltonian form as 
\begin{equation*}
   i\partial_{t}\Phi=(\alpha\cdot D)\Phi+m\beta\Phi.
\end{equation*}
Suppose $H_{m}(D)=\alpha\cdot D+m\beta$. We denote its propagator by $S_{m}(t)=e^{-itH_{m}(D)}$ and write $\Phi_{N}(t,x)=S_{m}(t)P_{N}\Phi_{0}(x)$ for a dyadic frequency component.

For $\omega\in\mathbb{S}^2$, define the massless spinor projections
\begin{equation*}
   P_{\omega}^{\pm}=\frac{1}{2}(I\pm \alpha\cdot\omega)
\end{equation*}
They are the spectral projections of $\alpha\cdot\xi$ when $\xi=|\xi|\omega$. If a matrix $\Gamma$ satisfies 
\begin{equation*}
\Gamma(\alpha\cdot\omega)=-(\alpha\cdot\omega)\Gamma
\end{equation*}
 for every $\omega\in\mathbb{S}^{2}$, then the bilinear scaler form $ B_{\Gamma}(z,w):=\langle \Gamma z, w\rangle_{\mathbb{C}^4}$ annihilates equal polarizations, i.e., $B_{\Gamma}(P_{\omega}^{\pm}z, P_{\omega}^{\pm}w)=0$. Thus
the Dirac null form
\begin{equation*}
    B_{\Gamma}(z,w)=B_{\Gamma}(P_{\omega}^{+}z, P_{\omega}^{-}w)+B_{\Gamma}(P_{\omega}^{-}z, P_{\omega}^{+}w)
\end{equation*}
We refer this property as the Dirac null structure. In particular, $\Gamma=\beta$ and $\Gamma=\beta\gamma^{5}$ generate the two basic quadratic spinor null forms considered in this paper.

\medskip
 
 The key point is that this algebraic decomposition matches the transport structure used in the spinorial div-curl lemma. Writing $\Phi=(\phi_1,\phi_2, \phi_3, \phi_4)^{T}$ and setting 
 \begin{equation*}
    U_{\Phi}=\begin{pmatrix}
                       \phi_{1}+\phi_{4}\\
                       \phi_{2}+\phi_{3}
                    \end{pmatrix}, \quad V_{\Phi}=\begin{pmatrix}
                                                                        \phi_{1}-\phi_{4}\\
                                                                        \phi_{2}-\phi_{3}
                                                                     \end{pmatrix},
 \end{equation*}
The Dirac equation becomes a pair of transport equations in the $x_{1}$ direction, which means the principal part transports $U_{\Phi}$ and $V_{\Phi}$ in opposite $x_1$-directions, while the transverse derivatives and the mass term couple the two modes. Integrating their squared norms over $y=(x_{2}, x_{3})$ gives two one-dimensional balance laws. The coupling term appears with opposite signs, so it only transfers charge between the two modes and cancels in the equation for their total density. The one-dimensional div-curl estimate then controls the mixed interaction $e(U_{\Phi})e(V_{\Psi})+e(V_{\Phi})e(U_{\Psi})$.
On the other hand, for the direction $\omega=e_{1}=(1,0,0)$, the projections $P_{\omega}^{+}$ and $P_{\omega}^{-}$ are, up to fixed isometries, exactly the $U$ and $V$ modes. If $\Gamma$ anticommutes with $\alpha\cdot\omega$, it exchanges these two eigenspaces, and the equal-polarization interaction vanish. 
Consequently, 
\begin{equation*}
    \Big| B_{\Gamma}(\Phi, \Psi) \Big|\lesssim_{\Gamma} |U_{\Phi}||V_{\Psi}|+|V_{\Phi}||U_{\Psi}|
\end{equation*}
Thus the algebraic null forms is precisely the mixed-mode interaction controlled by the balance laws. Rotation and angular localization then yield the dyadic bilinear estimate.This provides a direct link between
spinorial cancellation, conservation laws and bilinear spacetime integrability.

\medskip

A further feature appears in the massive problem. Let $\gamma^{5}=i\gamma^{0}\gamma^{1}\gamma^{2}\gamma^{3}$. The matrix
\[
  \Gamma_5:=\beta\gamma^5
\]
anticommutes not only with the spatial matrices $\alpha^j$, but also with
$\beta$. It therefore anticommutes with the full massive Dirac Hamiltonian $H_{m}(D)$.
In Fourier space, $H_{m}(\xi)$ has eigenvalues $\pm\lambda_{m}(\xi)$, whereas the normalized symbol $H_{m}(\xi)/\lambda_{m}(\xi)$ has eigenvalues $\pm1$.
For same-sign energy components at frequencies $K\ll m$, this stronger
algebraic property produces an additional factor $K/m$. This distinguishes
the pseudoscalar channel $\Gamma=\Gamma_5$ from the scalar channel $\Gamma=\beta$,
for which the mass term leaves a non-vanishing same-branch interaction.

\medskip

We now state the main estimates. Throughout, we write
\begin{equation*}
 L^{2}_{t,x}:=L^{2}([0,T]\times\mathbb{R}^{3})
\end{equation*}
For dyadic frequencies $(N_{1}, N_{2})$, define
\begin{equation*}
  \mathcal{C}_{m,T}(N_{1}, N_{2}):=\min\left\{ N_{*}\sqrt{2+mT},\  N_{*}^{3/2}T^{1/2}\right\},\quad
  N_{*}:=\min\{N_{1}, N_{2}\}. 
 \end{equation*}

\medskip

\begin{theorem}[Frequency-localized null-form estimates]
Let $m\geq 0$, and let $\Gamma\in\mathbb{C}^{4\times4}$ satisfy 
\begin{equation*}
\{\Gamma, \alpha\cdot\omega\}=0\  \ \text{for every} \ \ \omega\in\mathbb{S}^{2}.
\end{equation*}
Then any dyadic components $\Phi_{N_1}$ and $\Psi_{N_2}$ of free Dirac 
 solutions satisfy 
 \begin{equation*}
   \bigl\| B_{\Gamma}(\Phi_{N_1}, \Psi_{N_2})\bigr\|_{L^2_{t,x}} \lesssim \mathcal{C}_{m,T}(N_1, N_2)\
   \|\Phi_{N_1}(0)\|_{L^2_x}\|\Psi_{N_2}(0)\|_{L^2_{x}}.
 \end{equation*}
 Moreover, suppose that $\Gamma_1$ and $\Gamma_{2}$ satisfy the same anticommutation condition, and let $\Phi^{j}_{N_j}$, $j=0,1,2,3$, be dyadic components of the free Dirac solutions. Then
 \begin{equation*}
 \begin{aligned}
   &\left| \int_0^T\int_{\mathbb{R}^{3}}\left\langle B_{\Gamma_{1}}\big(\Phi^{1}_{N_1}, \Phi^2_{N_2}\big)\Gamma_2\Phi^3_{N_3},\Phi^0_{N_0}\right\rangle_{\mathbb{C}^{4}}dxdt\right|\\
   &\qquad \lesssim_{\Gamma_{1}, \Gamma_{2}}\mathcal{C}_{m,T}(N_1, N_2)\mathcal{C}_{m,T}(N_{3},N_{0})\prod_{j=0}^{3}\|\Phi^{j}_{N_j}(0)\|_{L^2_x}.
 \end{aligned}
\end{equation*}
\end{theorem}

\begin{theorem}[Low-frequency refinement in the pseudoscalar channel]
Assume $m>0$ and $\Gamma_{5}=\beta\gamma^{5}$. For $s\in\{+,-\}$, let $\Pi_{s}^{m}(D)$ denote the spectral projection of the massive Dirac Hamiltonian $H_{m}(D)$ onto the energy branch $s$, and set
\begin{equation*}
\Phi_{K,s}:=S_{m}(t)P_{K}\Pi_{s}^{m}(D)\Phi_0,\qquad \Psi_{K,s}:=S_{m}(t)P_{K}\Pi_{s}^{m}(D)\Psi_0.
\end{equation*}
Then,  in the low-frequency regime $K\ll m$,
\begin{equation*}
  \big\| B_{\Gamma_5}(\Phi_{K,s}, \Psi_{K,s}) \big\|_{L^2_{t,x}}\lesssim \min\left\{ K\sqrt{2+mT}, \ T^{1/2}\frac{K^{5/2}}{m}\right\}
  \|\Phi_{K,s}(0)\|_{L^2_x}\|\Psi_{K,s}(0)\|_{L^2_x}.
\end{equation*}
Thus, the same-branch interaction gains an additional factor $K/m$.
\end{theorem}

\medskip

The paper is organized as follows. In section 2, we derive the spinorial div-curl lemma from the component form of the Dirac equation. In section 3, we introduce the spinor projections and null forms and establish their precise relation to the $U-V$ decomposition. Section 4 proves the frequency-localized bilinear estimate and its cubic consequences. In section 5, we analyze the exact massive spectral projections and prove the refined same-branch estimate for $\Gamma_{5}=\beta\gamma^5$.
%\noindent\textbf{Sample theorem statement.}
%Assume the hypotheses of your problem are satisfied. Then insert the statement of your main theorem here.

\section{A spinorial div-curl lemma}
Throughout this section, $m\geq0$ and $\Phi=(\phi_1, \phi_2, \phi_3, \phi_4)^{T}$ denotes a smooth solution of the free Dirac eequation
\begin{equation*}
    i\partial_{t}\Phi=(\alpha\cdot D+m\beta)\Phi
\end{equation*}
In the standard Dirac representation fixed in Section 1, this equation reads 
\begin{empheq}[left=\empheqlbrace]{align}
   \partial_{t}\phi_{1}+\partial_{1}\phi_{4}-i\partial_{2}\phi_{4}+\partial_{3}\phi_{3}+im\phi_{1}&=0, \label{Dirac-2.1}\\
  \partial_{t}\phi_{2}+\partial_{1}\phi_{3}+i\partial_{2}\phi_{3}-\partial_{3}\phi_{4}+im\phi_{2}&=0, \label{Dirac-2.2}\\
  \partial_{t}\phi_{3}+\partial_{1}\phi_{2}-i\partial_{2}\phi_{2}+\partial_{3}\phi_{1}-im\phi_{3}&=0, \label{Dirac-2.3}\\
  \partial_{t}\phi_{4}+\partial_{1}\phi_{1}+i\partial_{2}\phi_{1}-\partial_{3}\phi_{2}-im\phi_{4}&=0.\label{Dirac-2.4}
\end{empheq}
Let $u_{1}=\phi_{1}+\phi_{4}$, $u_{2}=\phi_{2}+\phi_{3}$, $v_{1}=\phi_{1}-\phi_{4}$ and $v_{2}=\phi_{2}-\phi_{3}$.  Then we have
\begin{empheq}[left=\empheqlbrace]{align}
  (\partial_{t}+\partial_{1}) u_{1}&=-i\partial_{2}v_{1}+\partial_{3}v_{2}-imv_{1}, \label{Dirac-2.5}\\
  (\partial_{t}+\partial_{1}) u_{2}&=-\partial_{3}v_{1}+i\partial_{2}v_{2}-imv_{2}, \label{Dirac-2.6}\\
  (\partial_{t}-\partial_{1}) v_{1}&=i\partial_{2}u_{1}-\partial_{3}u_{2}-imu_{1}, \label{Dirac-2.7}\\
  (\partial_{t}-\partial_{1}) v_{2}&=\partial_{3}u_{1}-i\partial_{2}u_{2}-imu_{2}.\label{Dirac-2.8}
\end{empheq}
For any spinor field $\Phi=(\phi_1, \phi_2, \phi_3, \phi_4)^{T}$, set
\begin{equation*}
    U_{\Phi}=\begin{pmatrix}
                       u_1\\
                       u_2
                   \end{pmatrix}=\begin{pmatrix}
                                                      \phi_{1}+\phi_{4}\\
                                                     \phi_{2}+\phi_{3}
                                               \end{pmatrix}, \qquad   V_{\Phi}=\begin{pmatrix}
                                                                                                     v_1\\
                                                                                                     v_2
                                                                                           \end{pmatrix}=\begin{pmatrix}
                                                                                                                       \phi_{1}-\phi_{4}\\
                                                                                                                        \phi_{2}-\phi_{3}
                                                                                                                  \end{pmatrix}
\end{equation*}
Writing $x=(x_1, y)$, with $y=(x_2, x_3)$, define the transverse Dirac operator
\[
D_{y}:=\begin{pmatrix}
i\partial_{2} & -\partial_{3} \\
\partial_{3} & -i\partial_{2}
\end{pmatrix}
\]

\begin{lemma}
The operator $D_y$, with domain $H^{1}(\mathbb{R}^2; \mathbb{C}^2)$, is self-adjoint on $L^{2}(\mathbb{R}^{2}; \mathbb{C}^{2})$.
\end{lemma}

\begin{proof}
Under the Fourier transform in $y$, the operator $D_y$ has symbol 
\[
  s_y(\xi)=\begin{pmatrix}
                        -\xi_2 & -i\xi_3\\
                        i\xi_3 & \xi_2
                 \end{pmatrix}, \qquad  \xi=(\xi_2, \xi_3).
\]
Since 
\[ s^{*}_y(\xi)=s_y(\xi) \quad \text{and} \quad s^{2}_y(\xi)=|\xi|^2 I_2,
\]
 the corresponding Fourier multiplier is self-adjoint with domain $H^1(\mathbb{R}^2; \mathbb{C}^4)$.
\end{proof}

\medskip

Set $D_{y,m}:=D_y+imI_2$ and $D_{y,m}^{*}=D_y-imI_2$.
Adding and substracting (\ref{Dirac-2.5})--(\ref{Dirac-2.8}) gives the transport system
\begin{empheq}[left=\empheqlbrace]{align}
     & (\partial_{t}+\partial_{1})U_{\Phi}=-D_{y,m}V_{\Phi}, \label{Dirac-3.1}\\
      &(\partial_{t}-\partial_{1})V_{\Phi}=D_{y,m}^{*}U_{\Phi}. \label{Dirac-3.2}
\end{empheq} 
For a two-component field $W=W(t,x_1, y)$, define its transverse density by 
\begin{equation*}
   e(W)(t,x_1):=\int_{\mathbb{R}^2} |W(t,x_1,y)|^2 dy.
\end{equation*}
 We also set 
 \begin{equation*}
    F_{\Phi}(t,x_1):=2 \text{Re}\int_{\mathbb{R}^2}\left\langle D_{y,m} V_{\Phi}, U_{\Phi}\right\rangle_{\mathbb{C}^2}dy.
 \end{equation*}
 Taking the real part of the $L_y^2$ inner product of the first equation (\ref{Dirac-3.1}) with $U_{\Phi}$, and of the section equation (\ref{Dirac-3.2}) with $V_{\Phi}$, yields
\begin{empheq}[left=\empheqlbrace]{align}
     & \partial_{t} e(U_{\Phi})+\partial_{1} e(U_{\Phi})=-F_{\Phi}, \label{Dirac-3.3}\\
      &\partial_{t} e(V_{\Phi})-\partial_{1} e(V_{\Phi})=F_{\Phi}. \label{Dirac-3.4}
\end{empheq} 
 Here we used $\big\langle V_{\Phi}, D_{y,m}^{*}U_{\Phi}\big\rangle_{L_y^{2}}= \big\langle D_{y,m}V_{\Phi}, U_{\Phi}\big\rangle_{L_y^{2}}$.
 Define 
 \begin{equation*}
     E_{\Phi}:=e(U_{\Phi})+e(V_{\Phi}), \qquad M_{\Phi}:=e(U_{\Phi})-e(V_{\Phi}).
 \end{equation*}
 Adding equations (\ref{Dirac-3.3}) and (\ref{Dirac-3.4}), we obtain the one-dimensional conservation law
 \begin{equation}\label{conservation-law-1}
    \partial_t E_{\Phi}+\partial_1 M_{\Phi}=0.
 \end{equation}
 Moreover, 
 \begin{equation}\label{conservation-law-2}
     \int_{\mathbb{R}} E_{\Phi}(t,x_1) dx_1 =2\|\Phi(t)\|_{L_x^2}^2=2\|\Phi(0)\|_{L_x^2}^2
 \end{equation}

\medskip

We use the following one-dimensional div-curl estimate of Wang-Zhou \cite{WangZhou}.
\begin{lemma}[Wang-Zhou div-curl estimate]\label{Div-Curl}
   Suppose that  
   \[
       \begin{cases}
           \partial_{t} f^{11}+\partial_{x}f^{12}=G^{1},\\
           \partial_{t}f^{21}-\partial_{x}f^{22}=G^{2}.
       \end{cases}
   \]
  and $f^{11}, f^{12}, f^{21}, f^{22}\rightarrow 0$ as $|x|\rightarrow\infty$, where $(t,x)\in [0,T]\times\mathbb{R}$ and each $f^{ij}$ is a real-value function of $(t, x)$. Then we have
  \begin{equation}\label{div-curl}
  \begin{aligned}
     &\int_{0}^{T}\int_{-\infty}^{+\infty}f^{11}f^{22}+f^{12}f^{21}dxdt\lesssim
\Big(
\|f^{11}(0)\|_{L^1}
+\sup_{0\le t\le T}\|f^{11}(t)\|_{L^1} \\
&+\int_0^T \int_{-\infty}^{+\infty} |G^1|dxdt
\Big) 
\cdot
\Big(
\|f^{21}(0)\|_{L^1}
+\sup_{0\le t\le T}\|f^{21}(t)\|_{L^1}
+\int_0^T \int_{-\infty}^{+\infty} |G^2|dxdt
\Big).
\end{aligned}
  \end{equation}
 provided that the right hand side is bounded. 
\end{lemma}

\medskip

We now apply this estimate to the balance laws (\ref{Dirac-3.3}) (\ref{Dirac-3.4}) and (\ref{conservation-law-1}) above.

\begin{lemma}[Spinorial div-curl estimate]\label{Spin-div-curl}
Let $\Phi$ and $\Psi$ be free Dirac solutions with the same mass $m$ and Schwartz initial data. Then
\begin{equation}\label{Spinor-Div-Curl}
\begin{aligned}
    \int_{0}^{T}\int_{\mathbb{R}} &e(U_{\Phi})e(V_{\Psi})+e(V_{\Phi})e(U_{\Psi})dx_{1}dt\\
    & \lesssim \|\Psi(0)\|_{L_x^2}^{2}\Big(\|\Phi(0)\|_{L_x^{2}}^{2}+\int_{0}^{T}\int_{\mathbb{R}^{3}}|D_{y,m}V_{\Phi}| |U_{\Phi}|dxdt\Big)
\end{aligned}
\end{equation}
\end{lemma}
\begin{proof}
For free solutions $\Phi$ and $\Psi$, the balance laws give
\begin{equation*}
    \partial_{t} e(U_{\Phi})+\partial_{1} e(U_{\Phi})=-F_{\Phi}, \quad  \partial_{t} e(V_{\Phi})-\partial_{1} e(V_{\Phi})=F_{\Phi}.
\end{equation*}
and 
\begin{equation*}
   \partial_t E_{\Psi} +\partial_{1} M_{\Psi}=0.
\end{equation*}
First apply Lemma \ref{Div-Curl} with
  \begin{equation}\label{UV-1}
       \begin{cases}
         \partial_{t} e(U_{\Phi})+\partial_{1} e(U_{\Phi})=-F_{\Phi},\\
           \partial_{t} E_{\Psi}-\partial_{1}(- M_{\Psi})=0
       \end{cases} 
   \end{equation}
The corresponding source terms are $G^1=-F_{\Phi}$ and $G^2=0$, while
\begin{equation*}
     f^{11}f^{22}+f^{12}f^{21}=e(U_{\Phi})(E_\Psi-M_{\Psi})=2e(U_{\Phi})e(V_{\Psi})
\end{equation*}
Hence
\begin{equation}\label{div-curl-estimate-1}
   2\int_{0}^{T}\int_{\mathbb{R}} e(U_{\Phi})e(V_{\Psi}) dx_{1}dt
   \lesssim\sup_{0\leq t\leq T}\|E_{\Psi}(t)\|_{L_{x}^{1}}
   \Big(\sup_{0\leq t\leq T} \|e(V_{\Phi})\|_{L_{x}^1}+\|F_{\Phi}\|_{L_{t,x_1}^1}\Big)
\end{equation}
Next apply Lemma \ref{Div-Curl} with
  \begin{equation}\label{UV-2}
       \begin{cases}
          \partial_{t} e(V_{\Phi})-\partial_{1} e(V_{\Phi})=F_{\Phi},\\
           \partial_{t} E_{\Psi}+\partial_{1} M_{\Psi}=0.
       \end{cases}
   \end{equation}
 Now $G^1=0$, $G^2=F_{\Phi}$, and $f^{11}f^{22}+f^{12}f^{21}=e(V_{\Phi})(E_{\Psi}+M_{\Psi})=2e(V_{\Phi})e(U_{\Psi})$. Therefore, 
 \begin{equation}\label{div-curl-estimate-2}
   2\int_{0}^{T}\int_{\mathbb{R}} e(V_{\Phi})e(U_{\Psi}) dx_{1}dt
     \lesssim\sup_{0\leq t\leq T}\|E_{\Psi}(t)\|_{L_{x}^{1}}
   \Big(\sup_{0\leq t\leq T} \|e(U_{\Phi})\|_{L_{x}^1}+\|F_{\Phi}\|_{L_{t,x_1}^1}\Big)
\end{equation}
By (\ref{conservation-law-2}), we have
\begin{equation*}
     \sup_{0\leq t\leq T}\|E_{\Psi}(t)\|_{L_{x_1}^1}=2\|\Psi(0)\|_{L_{x}^2}^2,
\end{equation*}
 and
\begin{equation*}
    \sup_{0\leq t\leq T} \Big(\|e(U_{\Phi})(t)\|_{L_{x_1}^1}+e(V_{\Phi})(t)\|_{L_{x_1}^1}\Big)\lesssim \|\Phi(0)\|_{L_x^2}^2
\end{equation*}
 Moreover, 
 \begin{equation*}
      \|F_{\Phi}\|_{L_{t,x}^1}\leq 2\int_{0}^{T}\int_{\mathbb{R}^3} |D_{y,m}V_{\Phi}||U_{\Phi}| dxdt.
 \end{equation*}
 Combining (\ref{div-curl-estimate-1}) and (\ref{div-curl-estimate-2}) proves (\ref{Spinor-Div-Curl}).
  \end{proof}

\section{Spinor null structure and spinor null forms}
For $\omega\in\mathbb{S}^2$ and $\alpha^{j}=\gamma^{0}\gamma^{j}$,  set
\begin{equation}
  P_{\omega}^{\pm}=\frac{1}{2}(I\pm\alpha\cdot\omega), \quad \alpha\cdot\omega=\sum_{j=1}^{3}\omega_{j}\alpha^{j}
  \end{equation}
\begin{lemma}\label{self-adjoint}
For every $\omega\in\mathbb{S}^2$, the operators $P^{\pm}_{\omega}$ are mutually orthogonal self-adjoint projections. More precisely, 
\begin{equation*}
    \big(P_{\omega}^{\pm}\big)^{*}=P_{\omega}^{\pm},\quad  \big(P_{\omega}^{\pm}\big)^{2}=P_{\omega}^{\pm}, \quad
     P_{\omega}^{+}P_{\omega}^{-}=0,\quad P_{\omega}^{+}+P_{\omega}^{-}=I
\end{equation*}
\end{lemma}
\begin{proof}
Using $(\gamma^{0})^{*}=\gamma^{0}$, $(\gamma^{j})^{*}=-\gamma^{j}$ and $\gamma^{0}\gamma^{j}=-\gamma^{j}\gamma^{0}$, we obtain 
\begin{equation*}
   (\alpha^{j})^{*}=(\gamma^{0}\gamma^{j})^{*}=(\gamma^{j})^{*}(\gamma^{0})^{*}=-\gamma^{j}\gamma^{0}=\gamma^{0}\gamma^{j}=\alpha^{j}.
\end{equation*}
Moreover, the Clifford relations give 
\begin{equation*}
   \alpha^{j}\alpha^{k}+\alpha^{k}\alpha^{j}=-(\gamma^{j}\gamma^{k}+\gamma^{k}\gamma^{j})=2g_{jk}I.
\end{equation*}
Consequently, 
\begin{equation*}
   (\alpha\cdot\omega)^{*}=\alpha\cdot\omega,\quad 
   (\alpha\cdot\omega)^{2}=\frac{1}{2}\sum_{j,k=1}^{3}\omega_{j}\omega_{k}(\alpha^{j}\alpha^{k}+\alpha^{k}\alpha^{j})=|\omega|^{2}I=I.
\end{equation*}
Then the stated identities follow directly from $P_{\omega}^{\pm}=(I\pm\alpha\cdot\omega)/2$.
\end{proof}

If $\xi=|\xi|\omega$, then $\alpha\cdot\xi=|\xi|\alpha\cdot\omega$. Hence the projection operator $P_{\omega}^{\pm}=(I\pm\alpha\cdot\omega)/2$ is also the spectral projection of the massless Dirac symbol $\alpha\cdot\xi$ associated with the eigenvalue $\pm|\xi|$.

\begin{definition}[Dirac null structure]\label{null-structure}
Let $B:\mathbb{C}^{4}\times\mathbb{C}^{4}$ be a continuous sesquilinear form. We say that $B$ has a Dirac null structure if 
\begin{equation*}
  B(P_{\omega}^{\pm}z, P_{\omega}^{\pm}w)=0
\end{equation*}
for all $\omega\in\mathbb{S}^2$ and $z,w\in\mathbb{C}^{4}$. Equivalently,
\begin{equation*}
  B(z,w)=B(P_{\omega}^{+}z, P_{\omega}^{-}w)+B(P_{\omega}^{-}z,P_{\omega}^{+}w)
\end{equation*}
Thus only mixed-polarization interactions remain.
\end{definition}

\begin{proposition}\label{prop-3.3}
  Let $\Gamma\in\mathbb{C}^{4\times4}$ satisfy
 $\Gamma(\alpha\cdot\omega)=-(\alpha\cdot\omega)\Gamma$
  for every $\omega\in\mathbb{S}^{2}$. Then
  \begin{equation*}
     B_{\Gamma}(z,w):=\langle \Gamma z,w\rangle_{C^4}
  \end{equation*}
  possesses a Dirac null structure. 
\end{proposition}
\begin{proof}
    The condition $\Gamma(\alpha\cdot\omega)=-(\alpha\cdot\omega)\Gamma$ implies 
    \begin{equation*}
       \Gamma P_{\omega}^{\pm}=P_{\omega}^{\mp}\Gamma
    \end{equation*}
    Therefore, by Lemma \ref{self-adjoint}, 
    \begin{equation*}
      B_{\Gamma}(P_{\omega}^{\pm}z, P_{\omega}^{\pm}w)=\langle\Gamma P_{\omega}^{\pm}z, P_{\omega}^{\pm}w\rangle=\langle P_{\omega}^{\mp}\Gamma z, P_{\omega}^{\pm}w\rangle = \langle \Gamma z,  P_{\omega}^{\mp} P_{\omega}^{\pm}w\rangle=0
    \end{equation*}
    Expanding $z=P_{\omega}^{+}z+P_{\omega}^{-}z$ and $w=P_{\omega}^{+}w+P_{\omega}^{-}w$ gives
    \begin{equation*}
       B_{\Gamma}(z,w)=B_{\Gamma}(P_{\omega}^{+}z, P_{\omega}^{-}w)+B_{\Gamma}(P_{\omega}^{-}z, P_{\omega}^{+}w)
    \end{equation*}
    By Definition \ref{null-structure}, $ B_{\Gamma}(z,w)$ is a null form.
\end{proof}

\begin{corollary}[Quadractic null form]
    Assume $\beta=\gamma^{0}$ and $\gamma^{5}=i\gamma^{0}\gamma^{1}\gamma^{2}\gamma^{3}$. Then 
    \begin{equation*}
       B_{\beta}(z,w)=\langle \beta z, w\rangle_{\mathbb{C}^{4}}, \quad 
       B_{\beta\gamma^{5}}(z,w)=\langle \beta\gamma^{5}z,w\rangle_{\mathbb{C}^{4}}
    \end{equation*}
    are quadratic spinor null forms.
\end{corollary}
\begin{proof}
   For $j=1,2,3$, 
   \begin{equation*}
      \beta\alpha^{j}=\gamma^{0}\gamma^{0}\gamma^{j}=-\gamma^{0}\gamma^{j}\gamma^{0}=-\alpha^{j}\beta
   \end{equation*}
   Since $\gamma^{5}$ anticommutes with each $\gamma^{\mu}$, it commutes with $\alpha^{j}=\gamma^{0}\gamma^{j}$. Hence 
   \begin{equation*}
      (\beta\gamma^{5})\alpha^{j}=\beta\alpha^{j}\gamma^{5}=-\alpha^{j}\beta\gamma^{5}.
   \end{equation*}
   Thus the conclusion follows from Proposition \ref{prop-3.3}.
\end{proof}

\begin{definition}[Cubic spinor null form]\label{Cubic spinor null form}
   Let $\Gamma_{1}$ and $\Gamma_{2}$ satisfy $\Gamma_{i}(\alpha\cdot\omega)=-(\alpha\cdot\omega)\Gamma_{i}$. Define
   \begin{equation*}
      N_{\Gamma_{1}, \Gamma_{2}}(z,w,u):=B_{\Gamma_{1}}(z,w)\Gamma_{2}u.
   \end{equation*}
 When paired against a fourth spinor v, this expression factorizes as 
 \begin{equation*}
    \big\langle N_{\Gamma_{1}, \Gamma_{2}}(z,w,u), v\big\rangle_{\mathbb{C}^4}=\overline{B_{\Gamma_{1}}(z,w)}B_{\Gamma_{2}}(u,v).
 \end{equation*}
\end{definition}

\bigskip

We finally record the form of projectors in the direction used in the spinorial div-curl argument. Let 
$e_{1}=(1,0,0)$ and $Z=(z_{1}, z_{2}, z_{3}, z_{4})^{T}$ be any given spinor fields in $\mathbb{R}^{1+3}$. Define
\begin{equation*}
   U_{Z}:=\begin{pmatrix}
                   z_{1}+z_{4}\\
                   z_{2}+z_{3}
                \end{pmatrix},\quad   V_{Z}:=\begin{pmatrix}
                                                                   z_{1}-z_{4}\\
                                                                   z_{2}-z_{3}
                                                                 \end{pmatrix}, \quad J:=\begin{pmatrix}
                                                                                                             0 & 1\\
                                                                                                             1 & 0
                                                                                                        \end{pmatrix}
\end{equation*}
When $\omega=e_{1}=(1,0,0)$, we have
  \begin{equation*}
      \alpha\cdot e_{1}=\alpha^{1}=\gamma^{0}\gamma^{1}=
      \begin{pmatrix}
         0 & 0 & 0 & 1\\
         0 & 0 & 1 & 0\\
         0 & 1 & 0 & 0\\
         1 & 0 & 0 & 0
      \end{pmatrix}.
 \end{equation*}
Then for the spinor field $Z=(z_{1}, z_{2}, z_{3}, z_{4})$ and $W=(w_{1}, w_{2}, w_{3}, w_{4})$ with $z_{i}, w_{i}\in\mathbb{C}$, a direct calculation gives
 \begin{equation*}
    P_{e_{1}}^{+}Z=\frac{1}{2}(I+\alpha^{1})=\frac{1}{2}\big(z_{1}+z_{4}, z_{2}+z_{3}, z_{2}+z_{3}, z_{1}+z_{4}\big)^{T}=\frac{1}{2}(U_{Z}, JU_{Z})^{T}
 \end{equation*}
  \begin{equation*}
    P_{e_{1}}^{-}W=\frac{1}{2}(I-\alpha^{1})=\frac{1}{2}\big(w_{1}-w_{4}, w_{2}-w_{3}, -w_{2}+w_{3}, -w_{1}+w_{4}\big)^{T}=\frac{1}{2}(V_{W}, -JV_{W})^{T}
 \end{equation*}
Thus
\begin{equation}
   B(P_{e_{1}}^{+}Z, P_{e_{1}}^{-}W)+B(P_{e_{1}}^{-}Z, P_{e_{1}}^{+}W)\lesssim|U_{Z}||V_{W}|+|U_{W}||V_{Z}|
\end{equation}  
Consequently, every null form from Proposition \ref{prop-3.3} satisfies
\begin{equation}
    \big|B_{\Gamma}(Z, W)\big|\lesssim \frac{1}{2}\ \|\Gamma\|\ \Big(|U_{Z}| |V_{W}|+|V_{Z}| |U_{W}|\Big)
\end{equation}
This is the precise link between the spinor null structure and the spinorial div-curl estimate of the preceding section.

 \section{Proof of the Main Theorem 1.1}
 
 We first prove the bilinear estimate. Set 
 \begin{equation*}
   N_{*}:=\min\{N_1, N_2\}.
 \end{equation*}
 By symmetry, it suffices to consider the case $N_1\leq N_2$.
 Let 
 \begin{equation*}
     \rho:=\min\{N_1, T^{-1}\}.
 \end{equation*}
 Divide the annulus $\{|\xi|\sim N_1\}$ into smooth angular sectors of width about $\rho/N_{1}$.
 We denote the central direction of each sector $\kappa\in\Omega_{N_1,\rho}$ by $\omega_{\kappa}$.
 Thus, if $\omega_{\kappa}$ denotes the central direction of $\kappa$, then 
 \begin{equation*}
       \operatorname{supp}\chi_{\kappa}\subset \left\{ \xi: |\xi|\sim N_1, \angle(\xi, \omega_{\kappa})\lesssim \frac{\rho}{N_1}\right\}
       \quad\text{and}\quad  \#\Omega_{N_1,\rho}\lesssim \left(\frac{N_1}{\rho}\right)^{2}.
 \end{equation*}
 Define $\Phi_{N_1,\kappa}:=S_{m}(t)P_{N_1}\chi_{\kappa}(D)\Phi_{0}$. Then 
 \begin{equation*}
     \sum_{\kappa\in\Omega_{N_1,\rho}}\|\Phi_{N_1, \kappa}(0)\|_{L_x^2}^2\lesssim \|\Phi_{N_1}(0)\|_{L_x^2}^2,
 \end{equation*}
 where $\Phi_{N_1}=\sum_{\kappa\in\Omega_{N_1, \rho}}\Phi_{N_1, \kappa}$. Consequently, 
 \begin{equation}\label{decomposition}
    B_{\Gamma}(\Phi_{N_1}, \Psi_{N_2})=\sum_{\kappa\in\Omega_{N_1, \rho}}B_{\Gamma}(\Phi_{N_1, \kappa}, \Psi_{N_2})
 \end{equation}
 
 For $\kappa\in\Omega_{N_1, \rho}$. By rotational invariance, we may choose coordinates so that $\omega_{k}=e_1$. Writing 
 $x=(x_1, y)$, the Fourier support of $\Phi_{N_1, \kappa}$ then satisfies 
  \[
        |\xi_{y}| \lesssim \rho=\min\{N_{1}, T^{-1}\}.
 \]
The pointwise null-form estimate from section 3 gives 
\begin{equation*}
 \left|B_{\Gamma}(\Phi_{N_1,\kappa},\Psi_{N_2})\right|^{2}\lesssim_{\Gamma} |U_{\Phi_{N_1,\kappa}}|^2|V_{\Psi_{N_2}}|^2+|V_{\Phi_{N_1,\kappa}}|^2|U_{\Psi_{N_2}}|^2
 \end{equation*}
Applying the two-dimensional Bernstein inequality in the transverse variables, and using the spinorial div-curl estimate in Lemma \ref{Spin-div-curl}, we obtain 
\begin{equation}\label{bilinear-1}
\begin{aligned}
   &\|B_{\Gamma}(\Phi_{N_1,\kappa},\Psi_{N_2})\|_{L_{t,x}^2}^{2}\\
  &\lesssim_{\Gamma} \int_{0}^{T}\int_{\mathbb{R}}\int_{\mathbb{R}^2}  |U_{\Phi_{N_1,\kappa}}|^2\cdot|V_{\Psi_{N_2}}|^2+|V_{\Phi_{N_1,\kappa}}|^2\cdot|U_{\Psi_{N_2}}|^2 dydx_{1}dt\\
  &\lesssim_{\Gamma} \int_{0}^{T}\int_{\mathbb{R}} \|U_{\Phi_{N_1,\kappa}}\|_{L_y^{\infty}}^2\cdot\|V_{\Psi_{N_2}}\|_{L_y^2}^2+\|V_{\Phi_{N_1,\kappa}}\|_{L_y^{\infty}}^2\cdot\|U_{\Psi_{N_2}}\|_{L_y^2}^2 dx_{1}dt\\
  &\lesssim_{\Gamma} \rho^2 \int_{0}^{T}\int_{\mathbb{R}} \underbrace{\|U_{\Phi_{N_1,\kappa}}\|_{L_y^{2}}^2\cdot\|V_{\Psi_{N_2}}\|_{L_y^2}^2+\|V_{\Phi_{N_1,\kappa}}\|_{L_y^{2}}^2\cdot\|U_{\Psi_{N_2}}\|_{L_y^2}^2}_{e(U_{\Phi_{N_1,\kappa}})\cdot e(V_{\Psi_{N_2}})+e(V_{\Phi_{N_1,\kappa}})\cdot e(U_{\Psi_{N_2}})}
   dx_{1}dt\\
    &\lesssim_{\Gamma} \rho^2 \|\Psi_{N_2}(0)\|_{L_x^2}^{2}\Big(\|\Phi_{N_1,\kappa}(0)\|_{L_x^2}^{2}+\int_{0}^{T}\int_{\mathbb{R}^{3}}|D_{y,m}V_{\Phi_{N_1,\kappa}}|\cdot |U_{\Phi_{N_1,\kappa}}| dxdt\Big)
\end{aligned}
\end{equation}
Since the transverse Fourier support has size $O(\rho)$, then
\[
\|D_{y,m}V_{\Phi_{N_1,\kappa}}\|_{L_{y}^2}\leq (\rho+m)\|V_{\Phi_{N_1,\kappa}}\|_{L_{y}^2}.
\]
Using Cauchy-Schwarz and conservation of the Dirac charge, we obtain
\begin{equation}\label{bilinear-2}
\begin{aligned}
    &\int_{0}^{T}\int_{\mathbb{R}^{3}}|D_{y,m}V_{\Phi_{N_1,\kappa}}| |U_{\Phi_{N_1,\kappa}}| dxdt\\
    &\lesssim \int_{0}^{T}\int_{\mathbb{R}} (\rho+m)\|V_{\Phi_{N_1,\kappa}}\|_{L_{y}^2}
    \|U_{\Phi_{N_1,\kappa}}\|_{L_{y}^2} dx_{1}dt\\
    &\lesssim (\rho+m)\int_{0}^{T} \Big(\int_{\mathbb{R}}\|V_{\Phi_{N_1,\kappa}}\|_{L_{y}^2}^2dx_{1}\Big)^{1/2}
     \Big(\int_{\mathbb{R}}\|U_{\Phi_{N_1,\kappa}}\|_{L_{y}^2}^2dx_{1}\Big)^{1/2} dt\\
    &\lesssim (\rho+m)T \|\Phi_{N_1,\kappa}(0)\|_{L_x^2}^{2}
    % \lesssim (1+mT)\|\Phi_{N_1,\kappa}(0)\|_{L_x^2}^{2}
\end{aligned}
 \end{equation}
 Because $\rho T\leq 1$, it follows that 
 \[
     1+(\rho+m)T\leq 2+mT
 \]
 Combing inequalities (\ref{bilinear-1}) and (\ref{bilinear-2}), we have
 \begin{equation*}
    \|B_{\Gamma}(\Phi_{N_1,\kappa},\Psi_{N_2})\|_{L_{t,x}^2}\lesssim_{\Gamma} \rho\sqrt{2+mT}\|\Phi_{N_1,\kappa}(0)\|_{L^2}\cdot \|\Psi_{N_2}(0)\|_{L^2}
 \end{equation*}
 Summing over the sectors in (\ref{decomposition}) and using Cauchy-Schwarz, we obtain
 \begin{align*}
    &\|B_{\Gamma}(\Phi_{N_1},\Psi_{N_2})\|_{L_{t,x}^2}\\
    &\leq \sum_{\kappa\in\Omega_{N_1,\rho}} \|B_{\Gamma}(\Phi_{N_1,\kappa},\Psi_{N_2})\|_{L_{t,x}^2}\\
    &\lesssim_{\Gamma} \rho\sqrt{2+mT}\|\Psi_{N_2}(0)\|_{L_x^2}\sum_{\kappa\in\Omega_{N_1,\rho}} \|\Phi_{N_1,\kappa}(0)\|_{L_x^2}\\
    &\lesssim_{\Gamma} \rho\sqrt{2+mT}\|\Psi_{N_2}(0)\|_{L_x^2} \Big(\sum_{\kappa\in\Omega_{N_1,\rho}} 1^2\Big)^{1/2} \Big(\sum_{\kappa\in\Omega_{N_1,\rho}} \|\Phi_{N_1,\kappa}(0)\|_{L_x^2}^2\Big)^{1/2}\\
    &\lesssim_{\Gamma} N_{1}\sqrt{2+mT} \|\Phi_{N_1}(0)\|_{L_x^2} \|\Psi_{N_2}(0)\|_{L_x^2}
  \end{align*}
  Interchanging $\Phi$ and $\Psi$ gives the corresponding estimate when $N_2\leq N_2$. Hence
  \begin{equation}\label{null-estimate}
     \|B_{\Gamma}(\Phi_{N_1},\Psi_{N_2})\|_{L_{t,x}^2}\lesssim_{\Gamma} N_{*}\sqrt{2+mT} \|\Phi_{N_1}(0)\|_{L_x^2} \|\Psi_{N_2}(0)\|_{L_x^2}
  \end{equation}
 
 \medskip
 
 We next record the direct product estimate. Assuming again that $N_1\leq N_2$, the pointwise bound 
\[
    \big|B_{\Gamma}(\Phi_{N_1},\Psi_{N_2})\big| \lesssim_{\Gamma} |\Phi_{N_1}| |\Psi_{N_2}|
\]
and the three-dimensional Bernstein inequality  imply
\begin{equation}\label{direct-estimate}
\begin{aligned}
    \|B_{\Gamma}(\Phi_{N_1},\Psi_{N_2})\|_{L_{t,x}^2}^2 
    &\lesssim_{\Gamma} \int_{0}^{T}\int_{\mathbb{R}^3} |\Phi_{N_1}|^2|\Psi_{N_2}|^2 dxdt
    \lesssim_{\Gamma} \|\Phi_{N_1}(t,x)\|^2_{L_t^2 L_x^{\infty}}\|\Psi_{N_2}(t,x)\|^2_{L_t^{\infty} L_x^2}\\
    &\lesssim_{\Gamma} N_1^{3}T\|\Phi_{N_1}(0)\|^2_{L_x^2}\|\Psi_{N_2}(0)\|^2_{L_x^2}
%    &\lesssim \int_{0}^{T} \|\Phi_{N_1}(t,x)\|_{L_x^{\infty}}^{2} \|\Psi_{N_2}(t,x)\|_{L_x^2}^2 dt\\
   % &\lesssim N_{*}^{3} \int_{0}^{T} \|\Phi_{N_1}(t,x)\|_{L_x^2}^{2} \|\Psi_{N_2}(t,x)\|_{L_x^2}^2 dt\\
    %&\lesssim N_{*}^3 T\|\Phi_{N_1}(0)\|_{L_x^2}^2\|\Psi_{N_2}(0)\|_{L_x^2}^2
\end{aligned}
\end{equation}
By symmetry, we obtain
\begin{equation}\label{bilinear-estimate}
  \|B_{\Gamma}(\Phi_{N_1},\Psi_{N_2})\|_{L_{t,x}^2} \lesssim_{\Gamma} N_{*}^{3/2}T^{1/2}\|\Phi_{N_1}(0)\|_{L^2_{x}}\|\Psi_{N_2}(0)\|_{L_x^2}.
\end{equation}
Taking the minimum of  (\ref{null-estimate}) and (\ref{bilinear-estimate}) proves
\[
   \|B_{\Gamma}(\Phi_{N_1},\Psi_{N_2})\|_{L_{t,x}^2} \lesssim \min\big\{N_{*}\sqrt{2+mT}, \ N_{*}^{3/2} T^{1/2}\big\} \|\Phi_{N_1}(0)\|_{L_x^2} \|\Psi_{N_2}(0)\|_{L_x^2}.
\]

\bigskip

It remains to prove the estimate for the cubic null form. Let $\Phi_{N_j}^{j}, j=0,1,2,3$, be dyadic components of the free Dirac solutions. By the factorization in Definition \ref{Cubic spinor null form},
 \begin{align*}
    \left|\Big\langle B_{\Gamma_{1}}\big(\Phi^1_{N_1}, \Phi^2_{N_2}\big)\Gamma_{2}\Phi^3_{N_3}, \Phi^0_{N_0}\Big\rangle\right| = \left|B_{\Gamma_{1}}\big(\Phi^1_{N_1}, \Phi^2_{N_2}\big)\right|  \left| B_{\Gamma_{2}}\big(\Phi^3_{N_3},  \Phi^0_{N_0}\big)\right|
\end{align*}
Cauchy-Schwarz in spacetime and the bilinear estimate give
 \begin{align*}
 &  \left| \int_{0}^{T}\int_{\mathbb{R}} \Big\langle B_{\Gamma_{1}}\big(\Phi^1_{N_1}, \Phi^2_{N_2}\big)\Gamma_{2}\Phi^3_{N_3},  \Phi^0_{N_0} \Big\rangle dxdt \right| \\
   & \lesssim_{\Gamma}  \|B_{\Gamma_{1}}(\Phi_{N_{1}}^1, \Phi_{N_2}^2)\|_{L_{t,x}^2}\  \|B_{\Gamma_{2}}(\Phi_{N_{3}}^3,  \Phi^0_{N_0})\|_{L_{t,x}^2}\\
   & \lesssim_{\Gamma} C_{m,T}(N_1, N_2) C_{m,T}(N_3, N_0)\prod_{j=0}^{j=3} \|\Phi^j_{N_j}(0)\|_{L_x^2}
    \end{align*}
This completes the proof of Theorem~1.1.

\bigskip

\section{Proof of the Main Theorem 1.2}
\subsection{Massive spectral projections}
For $\xi\in\mathbb{R}^3$, set
\begin{equation*}
    H_m(\xi):=\alpha\cdot\xi+m\beta, \quad \lambda_{m}(\xi):=\sqrt{|\xi|^2+m^2}
\end{equation*}
The Clifford relations imply
\begin{equation*}
 \alpha^{i}\alpha^{j}+\alpha^{j}\alpha^{i}=0,\quad \beta\alpha^{j}+\alpha^{j}\beta=0, \quad \beta^2=I.
\end{equation*}
Consequently, we have $ H^{2}_{m}(\xi)=(\alpha\cdot\xi+m\beta)(\alpha\cdot\xi+m\beta)=(|\xi|^2+m^2)I$.
The eigenvalues of $H_{m}(\xi)$ are therefore $\pm\lambda_{m}(\xi)$, with spectral projections
\begin{equation}\label{projection}
    \Pi_{s}^{m}(\xi):=\frac{1}{2}\left(I+s\frac{H_{m}(\xi)}{\lambda_{m}(\xi)}\right), \quad s\in\{+,-\}
\end{equation}
These projections satisfy
\begin{equation*}
     \big(\Pi_{s}^{m}(\xi)\big)^{*}=\Pi_{s}^{m}(\xi),\quad \Pi_{s}^{m}(\xi)\Pi_{-s}^{m}(\xi)=0, \quad \Pi_{+}^{m}(\xi)+\Pi_{-}^{m}(\xi)=I.
\end{equation*}
Furthermore, we have
\begin{equation*}
     H_{m}(\xi)\Pi_{s}^{m}(\xi)=s\lambda_{m}(\xi)\Pi_{s}^{m}(\xi).
\end{equation*}
For later comparison, observe that
\begin{equation}\label{projection-operator-1}
   \Pi_{s}^{m}(\xi)= \frac{1}{2} \left(I+s \alpha\cdot \frac{\xi}{|\xi|}\right)+O\left(\frac{m}{|\xi|}\right), \quad \text{when}\ |\xi|\gg m,
\end{equation}
and 
\begin{equation}\label{projection-operator-2}
   \Pi_{s}^{m}(\xi)= \frac{1}{2} (I+s\beta)+O\left(\frac{|\xi|}{m}\right), \quad \text{when}\ |\xi|\ll m,
\end{equation}
Thus the massive projections approach the massless directional projections at high frequency, while at low frequency they approach the eigenspaces of $\beta$.

If $\Phi$ is a free Dirac solution, then $\hat{\Phi}(t,\xi)=e^{-itH_{m}(\xi)}\hat{\Phi}(0,\xi)$. Moreover,
\begin{align*}
    \hat{\Phi}(t,\xi)&=e^{-it\lambda_{m}(\xi)}\Pi_{+}^{m}(\xi)\hat{\Phi}(0,\xi)+e^{it\lambda_{m}(\xi)}\Pi_{-}^{m}(\xi)\hat{\Phi}(0,\xi)\\
    &=e^{-it\lambda_{m}(\xi)}\hat{\Phi}_{+}(0,\xi)+e^{it\lambda_{m}(\xi)}\hat{\Phi}_{-}(0,\xi)\\
    &=\hat{\Phi}_{+}(t,\xi)+\hat{\Phi}_{-}(t,\xi)
\end{align*}
Therefore, in the physical space, the solution $\Phi$ can be decomposed as 
\begin{align*}
     \Phi(t,x)&=\frac{1}{(2\pi)^3}\int_{\mathbb{R}^3} e^{ix\cdot\xi}\hat{\Phi}_{+}(t,\xi)d\xi+\frac{1}{(2\pi)^3}\int_{\mathbb{R}^3} e^{ix\cdot\xi}\hat{\Phi}_{-}(t,\xi)d\xi\\
     &=\Phi_{+}(t,x)+\Phi_{-}(t,x)
\end{align*}
The bilinear form decomposes according to the energy branches:
\begin{align*}
B_{\Gamma}(\Phi_{N_{1}}, \Psi_{N_{2}})&=\underbrace{B_{\Gamma}(\Phi_{N_{1},+}, \Psi_{N_{2},+}) + B_{\Gamma}(\Phi_{N_{1}, -}, \Psi_{N_{2},-})}_{ \text{same-branch interaction}}\\
&+ \underbrace{B_{\Gamma}(\Phi_{N_{1},+}, \Psi_{N_{2},-}) +B_{\Gamma}(\Phi_{N_{1},-}, \Psi_{N_{2},+})}_{ \text{oppesite-branch interaction}}
\end{align*}
where $\Phi_{N_{1},s}$ and $\Psi_{N_{2},s}$ have the explicit expressions 
\begin{equation*}
   \Phi_{N_1, s}(t,x)=\frac{1}{(2\pi)^3}\int_{\mathbb{R}^3} e^{ix\cdot\xi}e^{-its\lambda_{m}(\xi)}\chi_{N_1}(\xi)\Pi_{s}^{m}(\xi)\hat{\Phi}(0,\xi)d\xi,
\end{equation*}
\begin{equation*}
   \Psi_{N_2, s}(t,x)=\frac{1}{(2\pi)^3}\int_{\mathbb{R}^3} e^{ix\cdot\eta}e^{-its\lambda_{m}(\eta)}\chi_{N_2}(\eta)\Pi_{s}^{m}(\eta)\hat{\Psi}(0,\eta)d\eta.
\end{equation*}
Theorem~1.2 concerns the same-branch terms with $\Gamma=\Gamma_5$.

\subsection{\texorpdfstring{$\Gamma=\Gamma_5$}{Gamma = Gamma5} case}
\subsubsection{The pseudoscalar same-branch interactions}
Recall that $\Gamma_{5}:=\beta\gamma^5$. Since $\Gamma_5$ anticommutes with both $\alpha^j$ and $\beta$, it anticommutes with the full massive symbol:
\begin{equation*}
     \Gamma_5 H_m(\xi)=-H_{m}(\xi)\Gamma_5
\end{equation*}
It follows from (\ref{projection}) that
\begin{equation}
   \Gamma_5\Pi_{s}^{m}(\xi)=\Pi_{-s}^{m}(\xi)\Gamma_5
\end{equation}
Hence
\begin{equation}\label{operator-estimate-3}
   \Pi_{s}^{m}(\eta)\Gamma_5\Pi_{s}^{m}(\xi)=\Pi_{s}^{m}(\eta)\Pi_{-s}^{m}(\xi)\Gamma_5
   =\big( \Pi_{s}^{m}(\eta)-\Pi_{s}^{m}(\xi)\big) \Pi_{-s}^{m}(\xi)\Gamma_5
\end{equation}
Thus the same-branch interaction is controlled by the variation of the massive spectral projection.

\medskip

Set $A_{m}(\xi):=\frac{H_{m}(\xi)}{\lambda_{m}(\xi)}$. Then $\big( A_{m}(\xi)\big)^2=I$ and 
\begin{equation*}
  \Pi_{s}^{m}(\eta)-\Pi_{s}^{m}(\xi)=\frac{s}{2}\big(A_{m}(\eta)-A_{m}(\xi)\big).
\end{equation*}
Moreover, we have
\begin{equation*}
  \Big(A_{m}(\eta)-A_{m}(\xi)\Big)^2=A_{m}^{2}(\xi)+A_{m}^{2}(\eta)-\Big(A_{m}(\eta)A_{m}(\xi)+A_{m}(\xi)A_{m}(\eta)\Big)
\end{equation*}
\begin{align*}
   A_{m}^2(\xi)&=\frac{(\alpha\cdot\xi+m\beta)^2}{\lambda_{m}^2(\xi)}\\
  & =\frac{\sum_{i,j}\xi_{i}\xi_{j}\alpha^{i}\alpha^{j}+m\beta(\sum_{j}\xi_{j}\alpha^{j})+(\sum_{i}\xi_{i}\alpha^{i})m\beta+m^2\beta^2}{|\xi|^2+m^2}\\
   &=\frac{\sum_{i}|\xi_i|^{2}+m^2}{|\xi|^2+m^2}I=I
\end{align*}
\begin{align*}
  A_{m}(\xi) A_{m}(\eta)&+A_{m}(\eta)A_{m}(\xi)\\
   &=\frac{(\alpha\cdot\xi+m\beta)(\alpha\cdot\eta+m\beta)+(\alpha\cdot\eta+m\beta)(\alpha\cdot\xi+m\beta)}{\lambda_{m}(\xi)\lambda_{m}(\eta)}\\
   &=\frac{2\sum_{i}\xi_{i}\eta_{i}\alpha^i\alpha^i+2m^2\beta^2}{\lambda_{m}(\xi)\lambda_{m}(\eta)}I
   =\frac{2(\xi\cdot\eta+m^2)}{\lambda_{m}(\xi)\lambda_{m}(\eta)}I
\end{align*}
Consequently, 
\begin{equation}\label{operator-norm-1}
              \big\| \Pi_{s}^{m}(\eta)-\Pi_{s}^{m}(\xi)\big\|^2=\frac{1}{2}\underbrace{\left(1-\frac{\xi\cdot\eta+m^2}{\lambda_{m}(\xi)\lambda_{m}(\eta)}\right)}_{\text{RHS}}
\end{equation}
Let $\theta=\angle(\xi,\eta)$. The expression on the right-hand side can be written as 
\begin{equation}\label{operator-norm-2}
   \text{RHS}=\frac{|\xi||\eta|(1-\cos\theta)}{\lambda_{m}(\xi)\lambda_{m}(\eta)}+\frac{m^2 (|\xi|-|\eta|)^2}{\lambda_{m}(\xi)\lambda_{m}(\eta)\big(\lambda_{m}(\xi)\lambda_{m}(\eta)+|\xi||\eta|+m^2\big)}.
\end{equation}
Suppose now that $|\xi|\sim |\eta|\sim K$. We obtain from (\ref{operator-norm-1}) and (\ref{operator-norm-1}) that
\begin{equation}
        \big\| \Pi_{s}^{m}(\eta)-\Pi_{s}^{m}(\xi)\big\| \lesssim \frac{K}{\sqrt{K^2+m^2}}\sqrt{1-\cos\theta}+\frac{m}{K^2+m^2}\big| |\xi|-|\eta|\big|.
\end{equation}
In the low-frequency regime $K\ll m$, this gives the uniform estimate
\begin{equation}\label{operator-estimate-4}
    \big\| \Pi_{s}^{m}(\eta)-\Pi_{s}^{m}(\xi)\big\| \lesssim \frac{K}{m}.
\end{equation}
Combing (\ref{operator-estimate-3}) and (\ref{operator-estimate-4}), we conclude that whenever $|\xi|\sim|\eta|\sim K\ll m$, we have
\begin{equation}
    \|\Pi_{s}^{m}(\eta)\Gamma_5\Pi_{s}^{m}(\xi)\|\lesssim  \frac{K}{m}.
\end{equation}

\bigskip

We first consider the same-branch interactions $B_{\Gamma_5}(\Phi_{K,s}, \Psi_{K,s})$. In fact
\begin{align*}
     &B_{\Gamma_5}(\Phi_{K,s}, \Psi_{K,s})=\langle \Gamma_5\Phi_{K,s}, \Psi_{K, s} \rangle=(\Gamma_5\Phi_{K,s})^{*} \Psi_{K, s}\\
     &=\frac{1}{(2\pi)^{6}}\int_{\mathbb{R}^3}\int_{\mathbb{R}^3} e^{ix\cdot(\eta-\xi)}e^{it(\lambda_{m}(\xi)-\lambda_{m}(\eta))}
      \cdot \big\langle \Gamma_{5} \Pi_{s}^{m}(\xi)\hat{\Phi}_{K,s}(0,\xi), \Pi_{s}^{m}(\eta)\hat{\Psi}_{K, s}(0,\eta)\big\rangle d\xi d\eta \\
     &=\frac{1}{(2\pi)^{6}}\int_{\mathbb{R}^3}\int_{\mathbb{R}^3} e^{ix\cdot(\eta-\xi)} \big\langle \Gamma_{5} \Pi_{s}^{m}(\xi)\hat{\Phi}_{K,s}(t,\xi), \Pi_{s}^{m}(\eta)\hat{\Psi}_{K, s}(t,\eta)\big\rangle d\xi d\eta\\
     &=\frac{1}{(2\pi)^{6}}\int_{\mathbb{R}^3}\int_{\mathbb{R}^3} e^{ix\cdot(\eta-\xi)} \big\langle \Pi_{s}^{m}(\eta)\Gamma_{5} \Pi_{s}^{m}(\xi)\hat{\Phi}_{K,s}(t,\xi), \hat{\Psi}_{K, s}(t,\eta)\big\rangle d\xi d\eta
\end{align*}
Then the Fourier transform of the bilinear form satisfies
\begin{equation*}
     \big|\mathcal{F}\big\{B_{\Gamma_5}(\Phi_{K,s}, \Psi_{K, s})(t,\zeta)\big\}\big|\lesssim \frac{K}{m}\int_{\mathbb{R}^3} |\hat{\Phi}_{K,s}(t,\xi)||\hat{\Psi}_{K,s}(t, \xi+\zeta)|d\xi.
\end{equation*}
Plancherel's theorem and convolution Young inequality therefore give
\begin{equation}\label{Plancherel-estimate}
        \|B_{\Gamma_5}(\Phi_{K,s}, \Psi_{K, s})(t)\|_{L_x^2}\lesssim \frac{K}{m}\|\hat{\Phi}_{K, s}\|_{L_{\xi}^{2}}\|\hat{\Psi}_{K, s}\|_{L_\xi^1}\lesssim \frac{K^{5/2}}{m}\|\Phi_{K, s}\|_{L_x^2}\|\Psi_{K, s}\|_{L_x^2}.
\end{equation}
Here we used the fact that the Fourier support of $\Phi_{K, s}$ has volume $O(K^3)$. The free evolution is unitary on $L_x^2$, and hence
\begin{equation*}
     \|\Phi_{K, s}(t)\|_{L_x^2}=\|\Phi_{K,s}(0)\|_{L_x^2}, \quad \|\Psi_{K, s}(t)\|_{L_x^2}=\|\Psi_{K, s}(0)\|_{L_x^2}.
\end{equation*}
Integrating (\ref{Plancherel-estimate}) in time, we obtain
\begin{equation}\label{same-brach-estimate-1}
       \|B_{\Gamma_5}(\Phi_{K, s}, \Psi_{K, s})(t)\|_{L_{t,x}^2}\lesssim  T^{1/2} \frac{K^{5/2}}{m}\|\Phi_{K, s}(0)\|_{L_x^2}\|\Psi_{K, s}(0)\|_{L_x^2}.
\end{equation}
On the other hand, the null-form estimate proved in Section 4 gives
\begin{equation}\label{same-brach-estimate-2}
 \|B_{\Gamma_5}(\Phi_{K, s}, \Psi_{K, s})(t)\|_{L_{t,x}^2}\lesssim K\sqrt{2+mT}\|\Phi_{K, s}(0)\|_{L_x^2}\|\Psi_{K, s}(0)\|_{L_x^2}.
\end{equation}
Taking the minimum of (\ref{same-brach-estimate-1}) and (\ref{same-brach-estimate-2}), we conclude that 
\begin{equation*}
    \|B_{\Gamma_5}(\Phi_{K, s}, \Psi_{K, s})(t)\|_{L_{t,x}^2}\lesssim \min\left\{ K\sqrt{2+mT}, \frac{T^{1/2}K^{5/2}}{m}\right\}\|\Phi_{K, s}(0)\|_{L_x^2}\|\Psi_{K, s}(0)\|_{L_x^2}.
\end{equation*}
This proves Theorem~1.2.

\subsubsection{Comparison with the other interactions}
The preceding refinement is specific both to the pseudoscalar matrix and to the same energy branch. Indeed, for the opposite-branch interactions,
\begin{equation*}
    \Pi_{-s}^{m}(\eta)\Gamma_5\Pi_{s}^{m}(\xi)=\Pi_{-s}^{m}(\eta)\Pi_{-s}^{m}(\xi)\Gamma_{5}.
\end{equation*}
When $\eta=\xi$, this expression does not vanish. Consequently, no uniform factor $K/m$ is available for the opposite-branch pseudoscalar interaction.

\subsection{\texorpdfstring{$\Gamma=\beta$}{Gamma = beta} case}
The scalar matrix $\beta$ behaves differently. A direct calculation gives 
\begin{align*}
  \beta\Pi_{s}^{m}(\xi)&=\frac{1}{2}\Big(\beta+s\beta\frac{\alpha\cdot\xi+m\beta}{\lambda_{m}(\xi)}\Big)=\frac{1}{2}\Big(I-s\frac{\alpha\cdot\beta+m\beta}{\lambda_{m}(\xi)}\Big)\beta+s\frac{mI}{\lambda_{m}(\xi)}
  =\Pi_{-s}^{m}(\xi)\beta+\frac{m}{\lambda_{m}(\xi)}I
\end{align*}
Multiplying by $\Pi_{s}^{m}(\xi)$ on the left, we obtain
\begin{equation*}
     \Pi_{s}^{m}(\xi)\beta\Pi_{s}^{m}(\xi)=s\frac{m}{\lambda_{m}(\xi)}\Pi_{s}^{m}(\xi).
\end{equation*}
For $|\xi|\sim K\ll m$,
\begin{equation*}
    \frac{m}{\lambda_{m}(\xi)}\sim 1.
\end{equation*}
Thus the scalar same-branch interaction contains a non-vanishing leading term and does not gain the factor $\frac{K}{m}$. This explains the distinguish between the pseudoscalar and scalar channels stated in the introduction.

\bibliographystyle{plain}
\bibliography{references}

% -------------------------------------------------
\section*{Acknowledgments}
The author would like to thank Professor Yi Zhou for his academic lecture on the div-curl lemma.
% -------------------------------------------------

% -------------------------------------------------
% Bibliography
% -------------------------------------------------

% Option 1: BibTeX (recommended for a full paper).
% Create a file named references.bib, then uncomment the next two lines.
%
% \bibliographystyle{amsplain}
% \bibliography{references}

% Option 2: Write references directly in this file.
% Uncomment and edit the following block if preferred.
%
% \begin{thebibliography}{99}
%
% \bibitem{WangZhou}
% S.~Wang and Y.~Zhou,
% \emph{Title of the paper},
% Journal Name \textbf{volume} (year), pages.
%
% \end{thebibliography}

\end{document}